\documentclass[12pt]{amsart}

\usepackage{latexsym,amssymb,amscd,amsbsy,amsmath,amsfonts}
\usepackage{amsthm}
\usepackage[english]{babel}
\usepackage[latin1]{inputenc}
\newtheorem{Atheorem}{Theorem}[section]

\newtheorem{theorem}{Theorem}[section]
\newtheorem{lemma}[theorem]{Lemma}
\newtheorem{corollary}[theorem]{Corollary}
\newtheorem{proposition}[theorem]{Proposition}

\newtheorem{definition}[theorem]{Definition}
\newtheorem{question}{Question}
\newtheorem{remark}[theorem]{Remark}

\newenvironment{sproof}[1]
{\begin{proof}[#1]} {\end{proof}}

\newcommand{\N}{\mathbb N}

\newcommand{\Z}{\mathbb Z}

\newcommand{\Q}{\mathbb Q}
\newcommand{\C}{\mathbb C}

\newcommand{\M}{\mathcal{M}_2}
\newcommand{\Mn}{\mathcal{M}_n}
\newcommand{\Rn}{\mathcal{R}_n}
\newcommand{\Ru}{\mathcal{R}_1}
\newcommand{\Rd}{\mathcal{R}_2}
\newcommand{\Cp}{\mathcal{C}_p}
\newcommand{\Co}{\mathcal{C}_0}
\newcommand{\Cu}{\mathcal{C}}
\newcommand{\F}{\mathcal{F}}
\newcommand{\G}{G(\lambda)}
\newcommand{\GL}{\textnormal{GL}}
\newcommand{\Aut}{\textnormal{Aut}}
\newcommand{\Inn}{\textnormal{Inn}}
\newcommand{\V}{\operatorname{V}}

\newcommand{\br}[1]{\lbrack #1 \rbrack}
\newcommand{\Pres}[2]{\left\langle #1 \ \big\vert\  #2 \right\rangle}
\newcommand{\PresM}[2]{\ll #1 \ \big\vert\ #2 \gg}
\newcommand{\Ma}{M(\mathfrak{a})}
\newcommand{\Gb}{G(\mathfrak{b})}
\newcommand{\Ml}{M(\lambda)}
\newcommand{\Gl}{G(\lambda)}

\newcommand{\Gm}{G(\mu)}

\newcommand{\CB}{\operatorname{CB}}

\newcommand{\Zxy}{\Z \lbrack x^{\pm1},y^{\pm1} \rbrack}
\newcommand{\Zx}{\Z \lbrack x^{\pm 1} \rbrack}
\newcommand{\aid}{\mathfrak{a}}
\newcommand{\bid}{\mathfrak{b}}
\newcommand{\pid}{\mathfrak{p}}
\newcommand{\ok}{\overline{k}}
\newcommand{\ol}{\overline{l}}

\newcommand{\The}{\text{Th}_{\exists}}
\newcommand{\Thu}{\text{Th}_{\forall}}
\newcommand{\ufrak}{\mathfrak{U}}
\newcommand{\Mau}{M(\mathfrak{a}_{\mathfrak{U}})}
\newcommand{\au}{\mathfrak{a}_{\mathfrak{U}}}
\title{Limits of metabelian groups}
\author{Luc Guyot}
\date{\today}

\begin{document}%%
%%%%%%%%%%%%%%%%%%

\begin{abstract}
We describe the two-generated limits of abelian-by-(infinite cyclic) groups in the
space of marked groups using number theoretic methods. We also discuss universal equivalence of these limits.
\end{abstract}

\maketitle

\section{Introduction} \label{SecIntro}

We consider the class $\Cu$ of two-generated groups of the form $R \rtimes \Z$ where $R$ is an integral domain and $\Z$ acts on $R$ by multiplication of a unit. The class $\Cu$ contains for instance the lamplighter groups: $$(\Z/p\Z) \wr \Z=(\Z/p\Z)\br{x^{\pm1}} \rtimes_x \Z$$ for $p=0$ or $p$ prime and the soluble Baumslag-Solitar groups: $$BS(1,n)=\Z\br{1/n} \rtimes_n \Z,\, n \in \Z \setminus \{0\}.$$

Using results from algebraic number theory, we describe in this paper the closure of $\Cu$ in $\M$, the space of marked groups on two generators. Finitely generated groups with the same universal theory as free groups (resp. non-abelian free metabelian groups) are known to coincide with limits of free groups (resp. torsion-free lamplighter groups) in the space of finitely generated marked groups. Moreover, these limit groups were exhaustively listed for small ranks \cite{FGM+98,Chap97}. We are aiming for similar results for one-relator groups such as the Baumslag-Solitar groups and some of their natural generalizations.

A convenient way to define the space $\M$ is to fix a free group $F$ with basis $(a,b)$ and to consider the set of normal subgroups $N$ of $F$ endowed with the topology induced by the product topology of $2^{F}=\{0,1\}^{F}$ with discrete factors $\{0,1\}$. We will prefer to think of an element $N$ of $\M$ as a group $G$ marked by an ordered generating pair, namely the group $F/N$ marked by the image of $(a,b)$ via the quotient map. This way we will benefit from a tractable parametrization of the generating pairs of groups in $\Cu$ (Section \ref{SubSecGP}). 

 Let us give further definitions to state our main theorems. For a given $2$-generated group $G$, we denote by $\br{G}$ 
the set of all marked groups in $\M$ which are abstractly isomorphic to $G$. We denote by $\overline{\br{G}}$ the topological 
closure of $\br{G}$ in $\M$. For $\lambda \in \C \setminus \{0\}$, set $$\Gl=\Z\br{\lambda^{\pm 1}} \rtimes \Z$$ where $\Z$ 
acts by multiplication of $\lambda$ on $\Z\br{\lambda^{\pm 1}}$. 
We denote by $P_{\lambda}(x) \in \Z \br{x}$ the minimal polynomial of $\lambda$ over $\Z$ and set $P_{\lambda}(x)=0$ if $\lambda$ is transcendant. 
Observe that $G(\lambda)$ is isomorphic to $\Z \wr \Z$ in the latter case. Let $M$ be the free metabelian group on two generators, i.e. 
the quotient of $F$ by its second derived subgroup. Then $M$ fits in an extension 
$$1 \longrightarrow M'=\br{M,M} \longrightarrow M \longrightarrow \Z^2 \longrightarrow 1$$ where $M'$ is the derived subgroup of $M$. The conjugacy action of $\Z^2$ on $M'$ turns it into a module over the integral group ring of $\Z^2$. The latter ring is identified with the ring of Laurent polynomials $L=\Zxy$ and $M'$ is accordingly identified with $L$ by the Magnus Embedding Theorem (see \cite{Moc86} or \cite[Corollary 1.10]{Gup87}), with the commutator $\br{a,b}=aba^{-1}b^{-1}$ corresponding to the the unit polynomial $1 \in L$. It follows that any ideal $\aid$ of $L$ can be identified with a subgroup of $M'$ and it is indeed a normal subgroup of $M$, so that we can define
$$M(\aid) \doteqdot M/\aid.$$ For Laurent polynomials $P_1(x,y),P_2(x,y)\dots,$ denote by $(P_1(x,y),P_2(x,y),\dots)$ the ideal they generate in $L$ and put $ M(P_1(x,y),P_2(x,y),\dots) \doteqdot M(\aid)$ where $\aid=(P_1(x,y),P_2(x,y),\dots)$.
We denote by $\Phi_m(x) \in \Z \br{x}$ the $m$-th cyclotomic polynomial over the rationals.

\begin{Atheorem}[Prop. \ref{PropRoot} and Th. \ref{ThConv}] \label{ThA}
If $\lambda \in \C\setminus \{0,1\}$ is a primitive $m$-th root of unity then
$$\overline{\br{\Gl}}=\br{\Gl} \cup \br{M(\Phi_m(x),y-1)},$$
else
$$\overline{\br{\Gl}}=\br{\Gl} \cup \bigcup_{k>0} \br{M(P_{\lambda^k}(x))} \cup \br{M}.$$
\end{Atheorem}
An analogous result holds if we replace $\C$ by an arbitrary field of characteristic $p>0$ (Proposition \ref{PropRoot}).
The next results are sharper when $p=0$, so the introduction restricts to the class $\Co$ of two-generated groups $R \rtimes \Z$ with $R$ an integral domain of characteristic zero. The class $\Co$ was studied by Breuillard in the context of the slow growth of soluble groups and was related to the Lehmer conjecture in the following way: the conjecture holds if the exponential growth rates of non-virtually nilpotent groups in $\mathcal{C}_0$ are uniformly bounded away from $1$ \cite[Section 7]{Bre07}.
Recall that a non-zero algebraic number $\lambda$ is \emph{reciprocal} if $\lambda^{-1}$ is a conjugate of $\lambda$.
\begin{Atheorem}[Th. \ref{ThLimMa}] \label{ThB} Let $G$ be a $2$-generated group. 
\begin{itemize}
\item[$(i)$]
 The group $G$ is an accumulation point of $\mathcal{C}_0$ in $\M$ 
if and only if $G$ is isomorphic to $M(\aid)$ for some  prime ideal $\aid \subset \Zxy$ of the form $(P(x,y))$ or $(\Phi_m(x),\, y-1)$.
\item[$(ii)$]
The group $G$ is the limit of a sequence $(G(\lambda_n))$, where each $\lambda_n$ is a root of unity, if and only if $G$ is isomorphic to one of the groups: $M,\,\Z \wr \Z,\,M(\Phi_m(x)),\,M(\Phi_m(x),y-1) \quad (m \ge 1).$ In this case we call $G$ a $\Phi$-limit.
\item[$(iii)$] The group $G$ is the limit of a sequence $(G(\lambda_n))$, where each $\lambda_n$ is a reciprocal algebraic number, if and only if $G$ is either a $\Phi$-limit or it is isomorphic to $M(P(x,y))$ for some irreducible reciprocal polynomial $P$.
\end{itemize}
\end{Atheorem}

Given a subset $X \subset \C \setminus \{0\}$, we are also able to give a pointwise description of the closure of $\br{\Gl: \lambda \in X}$ in $\M$ provided that the set of algebraic units of $X$ is manageable. To be more precise, we use the following definitions.
The subset $X$ is said to be \emph{tame} if either $X$ is finite or $$\overline{\br{\Gl: \lambda \in X}}=\mathcal{L}(X) \cup \br{\Z \wr \Z}$$ with $\mathcal{L}(X)=\bigcup_{\lambda \in X} \overline{\br{\G}}$. It means that the limit points of $\br{\Gl: \lambda \in X}$ are the obvious ones, i.e. the limit points given by Theorem \ref{ThA}, plus the marked groups abstractly isomorphic to $\Z \wr \Z$ if $X$ is moreover infinite. Let $\lambda$ be an algebraic integer with conjugates $\lambda_1=\lambda,\lambda_2,\dots,\lambda_n$. Define
$\begin{array}{|c|} \hline \lambda \end{array} =\max_{1 \le i \le n}|\lambda_i|$ where $|\cdot|$ is the complex modulus. The algebraic integer $\lambda$ is an \emph{algebraic unit} if $\lambda \neq 0$ and $\lambda^{-1}$ is an algebraic integer.

\begin{Atheorem}[Th. \ref{ThTame} and Cor. \ref{CorBounded}] \label{ThC}
Let $X \subset \C \setminus \{0\}$. If there is $\epsilon>0$ such that $\begin{array}{|c|} \hline \lambda \end{array}>1+\epsilon$ for every algebraic unit $\lambda \in X$, then $X$ is tame. Therefore, any set of non-zero algebraic numbers with bounded degrees is a tame subset of $\C \setminus \{0\}$.
\end{Atheorem}

The set of limits associated to a tame subset of $\C \setminus \{0\}$ has a fairly simple topology which can be described by means of the Cantor-Bendixson rank (Remark \ref{RemTop}).

Finally, we characterize the $2$-generated groups with the same universal theory as $\Gl$. The computation of the Krull dimension of these groups, in the sense of Myasnikov and Romanovskiy, follows immediatly. The latter concept originates from algebraic geometry over groups where special normal subgroups with regards to discrimination replace the prime ideals of the classical ring theoretic setting \cite{BMR99}.
With the notations of Theorem \ref{ThA}, we set $M(\lambda)=M(\Phi_m(x),y-1)$ if $\lambda$ is a primitive $m$-th root of unity and $M(\lambda)=M(P_{\lambda}(x))$ else.

\begin{Atheorem}[Th. \ref{ThUniv} and Cor. \ref{CorKrull}] \label{ThD}
Let $\lambda \in \C \setminus \{0,1\}$ and let $G$ be a non-abelian $2$-generated group. The group $G$ has the same universal theory as $\Gl$ if and only if $G$ is isomorphic to $\Gl$ or $\Ml$.
Consequently, the Krull dimension of $\Ml$ is $1$.
\end{Atheorem}

Theorem \ref{ThD} can be viewed as generalization of \cite[Corollary 5.3]{Chap97} stating that $M$ and $\Z \wr \Z$ are, up to abstract isomorphism, the only $2$-generated groups with the same universal theory as a non-abelian free metabelian group.

The paper is organized as follows. Section \ref{SecNotAndBack} collects notations, definitions and elementary facts that we will use without further notice. Section \ref{SecLimits} is devoted to the proofs of Theorems \ref{ThA},\ref{ThB} and \ref{ThC}. First, we relate convergence of groups in $\Cu$ to convergence of prime ideals with respect to the patch topology (Subsection \ref{SubSecFrom}). Second, we establish a convergence criterion based on absolute values (Subsection \ref{SubSecCrit}). This criterion is crucial in the proofs of Theorem \ref{ThA} (Subsection \ref{SubSecSingle}) and Theorem \ref{ThC} (Subsection \ref{SubSecTame}). We prove Theorem \ref{ThB}, using Theorem \ref{ThA} and results of Schinzel on cyclotomic factors of lacunary polynomials (Subsection \ref{SubSecClosure}). Section \ref{SecUniv} is devoted to the proof of Theorem \ref{ThD}, which relies on Theorem \ref{ThC} and results of Champetier and Guirardel tying logic and convergence in the space of marked groups.

\section{Notations and background} \label{SecNotAndBack}

\subsection{Orbits of generating tuples and the space of marked groups} \label{SubSecOrbits}

Let $G$ be a group generated by $n$ elements and let $F_n=F(x_1,\dots,x_n)$ be the free group with basis $(x_1,\dots,x_n)$. Let $\V(G,n)$ be the set of \emph{generating $n$-vectors} of $G$, i.e. the set of $n$-tuples generating $G$.
We consider two natural actions on $\V(G,n)$. The first one is the diagonal left action of $\Aut(G)$ on  $\V(G,n)$ defined by $$\phi \cdot (g_1,\dots,g_n)=(\phi(g_1),\dots,\phi(g_n))$$ The second one is the right action of $\Aut(F_n)$, the group of automorphisms of $F_n$, through Nielsen transformations. For $\psi \in \Aut(F_n)$, let $w_i(x_1,\dots,x_n)=\psi^{-1}(x_i)$, with $i=1,\,2,\dots,n$. For $S=(g_1,\dots,g_n) \in \V(G,n)$ the action of $\psi$ on $S$ is defined by $$S \cdot \psi=(w_1(g_1,\dots,g_n),\dots,w_n(g_1,\dots,g_n))$$

 There is a convenient way to describe these two actions. Observe first that 
$\operatorname{V}(G,n)$ identifies with the set $\operatorname{Epi}(F_n,G)$ 
of surjective homomorphisms from $F_n$ onto $G$ through $S \mapsto p_S$ where $p_S=p_{(g_1,\dots,g_n)}$ is defined by $p_S(x_i)=g_i$.
The correponding actions read $$\phi \cdot p=\phi \circ p$$ for $\phi \in \Aut(G)$ and $$p \cdot \psi=p \circ \psi$$ for $\psi \in \Aut(F_n)$, which makes clear that the right action is well-defined and commutes with the left one. 
The set $\br{G}_n$ of $\Aut(G)$-orbits of generating $n$-tuples corresponds bijectively to the set of \emph{$G$-defining normal subgroups of $F_n$}, i.e., the normal subgroups $N \lhd F_n$ such that $F_n/N$ is isomorphic to $G$. It also corresponds to the set of \emph{$n$-generated marked groups} isomorphic to $G$. A $n$-generated marked group $(G,S)$ is a group $G$ marked by a generating $n$-tuple $S \in G^n$, defined up to isomorphisms respecting the markings. The set $\mathcal{M}_n$ of $n$-generated marked groups has a natural topology \cite{Cham00,ChGu05}. Here is the commonly used definition:

\begin{definition}
The Chabauty topology on $\Mn$ is the topology induced by the product topology of $2^{F_n}$ on $\{N \lhd F_n\}$ through the above identification. 
\end{definition}
The space $\Mn$ is then a totally discontinuous compact metrizable topological space.
It was first introduced by Gromov and Grigorchuk in the context of growth \cite{Grom81,Gri84}. (The general idea of topologizing sets of groups goes back to Mahlers and Chabauty \cite{Mah46,Chab50}; the interested reader should consult \cite{Har08} for a thorough account.) It was further used to prove both existence and abundance of groups with exotic properties \cite{Ste96, Cham00} and turned to be a remarkably suited framework for the study of Sela's \emph{limit groups} \cite{ChGu05}.
Groups whose isomorphism classes accummulate on free groups were studied by Ol'shanskii and Sapir \cite{OS09} in connection with percolation theory \cite{BNP09}. Limits of Thompson's group $F$ were investigated in \cite{Zar07} and free limits were exhibited in \cite{AST09,Bri09}. Limits of Baumslag-Solitar groups endowed with their standard generating pairs were studied in \cite{Sta06a,GS08}.

The isolated points in the space of marked groups were investigated in \cite{CGP07} and an isolated group was used by de Cornulier to answer a question of Gromov  concerning the existence of sofic groups which do not arise as limits of amenable groups \cite{Cor09b}. Neithertheless, very little is known on the topological type of $\Mn$ and so far the subspace of metabelian groups focus most attention \cite{Cor09a,CG10}. The box-counting dimension of the space of marked groups on $n$ generators is infinite if $n \ge 2$ \cite{Guy07} and its Haudorff dimension is non-zero \cite{GS2}.

 The countable compact space $\Rn$ of marked rings on $n$ invertible generators is defined similarly, considering the so-called \emph{patch topology} 
on the set of ideals of $L_n=\Z \br{x_1^{\pm1},\dots,x_n^{\pm1}}$, i.e. the topology induced by the product topology of $2^{L_n}$. This topology was investigated
 in \cite{Hoc69,HK91,CN03} and is related to the study of the subspace of metabelian groups in $\Mn$ \cite{Cor09,Cor09a}. 

One may keep in mind the
following characterization of convergent sequences \cite[Pr. 1]{Sta06a}.

\begin{lemma}\label{lmecv}
Let $(G_i,S_i)$ (resp. $(R_i,\Sigma_i)$) be a sequence of marked groups in $\Mn$ (resp. of marked rings in $\Rn$). The sequence $(G_i,S_i)$ (resp. $(R_i,\Sigma_i)$) converges if and only if for any $w \in F_n$ (resp. $W \in L_n$), we have either $w(S_i) =1$ in $G_i$ (resp. $W(\Lambda_i)=0$ in $R_i$) for
$i$ large enough, or $w(S_i)\neq 1$ in $G_i$ (resp. $W(\Lambda_i) \neq 0$ in $R_i$) for $i$ large enough.
\end{lemma}

It is important to not that the abstract isomorphism class $\br{G}_n$ of $G$ in $\Mn$ needs not to be closed, as shows for instance the following convergence $\lim_{i \to \infty} (\Z, (1,i,\dots,i^{n-1}))=(\Z^n, (e_j)_{1 \le j \le n}) \in \Mn$ \cite{ChGu05} where $(e_j)_{1 \le j \le n}$ stands for the canonical basis of $\Z^n$.
Let us make a brief digression on the problem of understanding closed isomorphism classes, looking first at groups with finitely many isomorphic copies in $\mathcal{M}_n$ (i.e. $\br{G}_n$ is finite), or equivalently, with finitely many $\Aut(G)$-orbits of generating $n$-tuples. If $G$ is finite, then it has clearly finitely many $\Aut(G)$-orbits of generating $n$-tuples for every $n \ge d(G)$, the minimal number of generators of $G$. Let $G$ be the $n$-generated (relative) free group of a given variety of groups. Assume moreover that $G$ is Hopfian, i.e. $G$ is not isomorphic to any of its proper quotient. It is then immediate that $G$ has only one $\Aut(G)$-orbit of generating $n$-tuples. In particular, we have $\br{\Z^2}_2=\{\Z^2\}$ and $\br{M}_2=\{M\}$.
\begin{question}
Let $G$ be finitely generated group such that $\br{G}_n$ is finite. Does $G$ contains a relative free group as a subgroup of finite index?
\end{question}
For every $(G,S) \in \Mn$ (resp. $N \lhd F_n$) and every $\psi \in \Aut(F_n)$, define $(G,S) \cdot \psi=(G,S \cdot \psi)$ (resp. $\psi \cdot N=\psi(N)$).
The right action of $\Aut(F_n)$ on $\Mn$ is continuous \cite{Cham00} with kernel $\Inn(F_n)$ \cite{Lub80}. There is an analoguous right action of $\GL(n,\Z)$ on $\Rn$ which is continuous and faithful. In order to describe this action, we introduce the following notations. For $\alpha=(\alpha_1,\dots,\alpha_n)\in \Z^n$ and $A \in \GL_n(\Z)$, set $x^{\alpha}=x_1^{\alpha_1} \cdots x_n^{\alpha_n},\,A \cdot x^{\alpha}=x^{A \alpha }$ and extend the action linearly to $L_n$. Given $A \in \GL_n(\Z)$, set $W_i(x_1,\dots,x_n)=A \cdot x_i$ for $i=1,\dots,n$. Let $R$ be a ring generated by $n$ of its invertible elements and let 
$\Lambda=(\lambda_1,\dots,\lambda_n) \in (R^{\times})^{n}$ a generating $n$-tuple. For 
$(R,\Lambda) \in \Rn$, define $$(R,\Lambda) \cdot A =(R, \Lambda \cdot A)$$ with $\Lambda \cdot A=(W_1(\lambda_1,\dots,\lambda_n),\dots,W_n(\lambda_1,\dots,\lambda_n))$. Let $\aid$ be an ideal of $L_n$, $\psi \in \Aut(F_n)$ and let $A(\psi) \in \GL_n(\Z)$ be the automorphism of $\Z^n$ induced by $\psi$. We set $$\begin{array}{c}
\psi \cdot \aid=A(\psi) \cdot \aid,\\
(R,\Lambda) \cdot \psi=(R,\Lambda) \cdot A(\psi).\end{array}$$

 Note that left multiplication by $\psi$ defines a ring automorphism 
of $L_n$, and hence preserves irreducibility of polynomials, as well as primeness and heights of ideals in $L_n$. Observe also that right multiplication by $\psi$ preserves the group isomomorphism relation on $\Mn$ and the ring isomorphism relation on $\Rn$.

Let us consider $(G,S) \in \mathcal{M}_m$ and $(H,T) \in \Mn$. If $G$ and $H$ are abstractly isomorphic, then there is neighborhood $U$ of $(G,S)$ in $\mathcal{M}_m$, a neighborhood $V$ of $(H,T)$ in $\Mn$ and an homeomorphism $f:U \longrightarrow V$ preserving the isomorphism relation such that $f(G,S)=(H,T)$ \cite[Lemma 1]{CGP07}. Hence, for any set $X$ of abstract isomorphism classes, the property of being a limit of groups in $X$ does not depend on the marking. This is the reason why Theorem \ref{ThB} is formulated without referring to any marking.

\subsection{Generating pairs} \label{SubSecGenPairs}
Let $A$ be an abelian group, $C$ an infinite cyclic group generated by $a$ and $G$ the central term of an exact sequence $$1 \longrightarrow A \longrightarrow G \stackrel{\sigma}{\longrightarrow} C  \longrightarrow 1.$$
The action of $C$ on $A$ by conjugacy extends linearily to $\Z \br{C}$, turning $A$ into a $\Z \br{C}$-module. Let $R$ be the quotient of $\Z\br{C}$ by the annihilator of $A$. If $G$ is $2$-generated, it is easy to check that $A$ is a cyclic $\Z\br{C}$-module. In this case, let $b$ be a generator of $A$ and let $\lambda$ be the image of $a$ by the natural map from $C$ into $R$. We write then 
$G=R \rtimes \Z$, identifying $b$ with the unity of $R$ and $a$ with the unity of $\Z$. We assume that $\lambda \neq 1$, i.e. $G$ is not abelian. For $k,k' \in \Z$, $g,h \in G$ with $g=(r,k)$ and $g'=(r',k')$, define
$$
 \delta(k)=\frac{\lambda^k-1}{\lambda-1},\quad
 \Delta(g,h)=r' \delta(k)-r \delta(k'),\quad
 \sigma(g,h)=(k,k').
$$ 
Let $k,k'$ be coprime integers. As $\sigma(S) \cdot \psi=\sigma(S \cdot \psi)$ for every $S \in G^2,\,\psi \in \Aut(F_2)$ and $\sigma(a,b)=(1,0)$, there exists a generating pair $S$ of $G$ such that $\sigma(S)=(k,k')$.  

Determining limits of marked groups in $\Cu$ will be facilitated by:
\begin{proposition}\cite[Proposition 1]{Guy1} \label{PropDelta}
 Let $G=R \rtimes \Z$ be a $2$-generated abelian-by-(infinite cyclic) group. Let $S \in G^2$ and let $S',S''$ be two generating pairs of $G$.
\begin{itemize}
 \item[$(i)$] The pair $S$ generates $G$,
if and only if $\sigma(S)$ is a pair of coprime integers and $\Delta(S)$ is a unit of $R$.
\item[$(ii)$] If $\sigma(S')=\sigma(S'')$ then there is an automorphism of $G$ such that $\phi \cdot S'=S''$.
\item[$(iii)$] Let $\lambda \in \C \setminus \{0,1\}$. If there is an automorphism of $G=\Gl$ such that $\phi \cdot S'=S''$ then $\sigma(S')= \epsilon \sigma(S'')$ with $\epsilon \in \{\pm 1\}$.
\end{itemize}
\end{proposition}

\begin{remark} \label{RemRec}
We note that $\Gl$ is isomorphic to $G(\lambda^{- 1})$ as a marked group if and only if $\lambda$ is either a reciprocal algebraic number or a transcendant number over the rationals. For such a $\lambda$, the $\epsilon$ of assertion $(iii)$ takes its two possible values. Else $\epsilon$ must be one.
\end{remark}

Understanding the isomorphism relation on the closure of $\Cu$ in $\M$ will be eased by:

\begin{proposition}  \cite[Theorem 2.4]{Bru74} and \cite[Theorem 4.10]{Dun64} (see also \cite[Proposition 2]{Guy1}) \label{PropTsys}
 Let $G$ be a $2$-generated group. Assume that $G$ is either abelian-by-(infinite cyclic) or metabelian with abelianization isomorphic to $\Z^2$.
Then $\Aut(F_2)$ acts transitively on $\br{G}_2$.
\end{proposition}

\subsection{Rings, polynomials and absolute values} \label{SubSecRings}
We only consider commutative rings with unity. For any such ring $R$, we denote by $R^{\times}$ its group of units. For $p$ prime or $p=0$, we denote by $\Z_p$ the ring $\Z/p\Z$ of intergers modulo $p$ (thus $\Z_0=\Z$) and by $\Z_p \br{x_1^{\pm 1},\dots,x_n^{\pm 1}}$ the ring of Laurent polynomials over $n$ variables with coefficient in $\Z_p$. In section \ref{SecLimits}. We will use the fact that $L_n=\Z \br{x_1^{\pm 1},\dots,x_n^{\pm 1}}$ is an Hilbert ring of Krull dimension $n+1$, i.e. each of its prime ideals is an intersection of maximal ideals \cite[Theorem 4.19]{Eis95} and the maximal length of an ascending chain of prime ideals is $n+1$. A polynomial $P(x_1,\dots,x_n)$ is said to be \emph{reciprocal} if $P(x_1^{-1},\dots,x_n^{-1})=\pm x_1^{\alpha_1}\dots x_n^{\alpha_n} P(x_1,\dots,x_n)$ for some $\alpha_1,\dots,\alpha_n \in \Z$.
Let $R \subset R'$ be rings. An element $\lambda \in R'$ is said to be \emph{algebraic over $R$} if there is a polynomial $P$ with coefficients in $R$ such that $P(\lambda)=0$ holds in $R'$.  Else $\lambda$ is \emph{transcendant over $R$}. Assume $R=\Z \subset \C=R'$ and let $\lambda \in \C$ be algebraic over $\Z$. The \emph{minimal polynomial of $\lambda$ over $\Z$} is the unique irreducible polynomial $P_{\lambda}(x) \in \Z\br{x}$ vanishing at $\lambda$ whose leading coefficient is positive. The algebraic number $\lambda$ is an \emph{algebraic integer} (resp. \emph{reciprocal algebraic number}) if $P_{\lambda}(x)$ is monic (resp. reciprocal). If both $\lambda$ and $\lambda^{-1}$ are algebraic integers, then $\lambda$ is an \emph{algebraic unit}. 

Let $K=\Q(\lambda)$ be a number field. For every prime ideal $\pid$ of the ring $\mathcal{O}(K)$ of algebraic integers of $K$, define the \emph{non-archimedean normalized absolute value} $|\cdot|_{\pid}$ by $|x|_{\pid}=(\text{N}\pid)^{-\text{ord}_{\pid}(x)}$ if $x\neq 0$  and $|0|_{\pid}=0$, where $\text{N}\pid$ denotes the number of elements $\mathcal{O}(K)/\pid$ and $\text{ord}_{\pid}(x)$ denotes the order of $x$ at $\pid$. Recall that $|x|_{\pid}=1$ holds for all prime ideal $\pid$ if and only if $x$ is an algebraic unit, i.e. $x \in \mathcal{O}(K)^{\times}$ (see \cite[Chapters I and II]{Lan94}).

For every prime $p$, we fix a transcendental extension $K_p$ of an algebraic closure of the field with $p$ elements. We set $K_0=\C$. 

\subsection{Cantor-Bendixson rank} \label{SubSecCB}
We deal mainly with countable compact spaces, namely the space of metablian groups on two generators and the spaces $\Rn\, (n=1,2)$. Such spaces can be entirely described by means of the \emph{Cantor-Bendixson rank} \cite{MS20}. Let us give its definition. If $X$ is a
topological space, define its derived subspace $X^{(1)}$ as the subset of accumulation
points in $X$. Iterating over ordinals
$$X^{(\alpha+1)} = (X^{(\alpha)})^{(1)},\, X^{(\alpha)} = \bigcap_{\beta<\alpha} X^{(\beta)} \text{ for limit } \alpha,$$
we have a non-increasing family $X^{(\alpha)}$ of closed subsets. The the Cantor-Bendixson rank $\CB(X)$ of $X$ is the least $\alpha$ such that
$X^{(\alpha)} = X^{(\alpha+1)}$.
Moreover if $x \in X$, we write
$$\CB_X(x) = \sup\{\alpha : x \in X^{(\alpha)}\}$$
if this supremum exists, in which case it is a maximum.
For instance, the Cantor-Bendixson rank of the space of metabelian marked groups on $n$ generators is $\omega^n \cdot(n-1)+1$ \cite[Theorem 1.4]{Cor09a}, where $\omega=(\N,\le)$ is the first infinite ordinal. The space $\operatorname{Spec}(L_n) \subset \Rn$ of prime ideals has Cantor-Bendixson rank $n+2$ \cite[Corollary 3.23]{CN03}. Moreover, for every prime ideal $\pid \subset L_n$ we have 
$$
\CB_{\operatorname{Spec}(L_n)}(\pid)=h(\pid)
$$
where $h(\pid)$ denotes the \emph{height of $\pid$}, i.e. the maximal length of an ascending chain of prime ideals ending at $\pid$.

\subsection{Ultrafilters} \label{SubSecUltra}
An ultrafilter $\ufrak$ on $\N$ is a collection of subsets of $\N$ such that:
\begin{itemize}
 \item  for all $I,J \in \ufrak$, we have $I \cap J \in \ufrak$,
\item for all $I \in \ufrak, J \subset \N$, if $I \subset J$ then $J \in \ufrak$,
\item for all $I \subset \N$, either $I \in \ufrak$ or $\N \setminus I \in \ufrak$, but not both.
\end{itemize}
 An ultrafilter is said to be \emph{free} if it is not \emph{principal}, i.e. if it is not of the form $\{A \subset \N: n\ \in A\}$ for some $n \in \N$.
We use free ultrafilters on $\N$ to pick arbitrary limit points of a countable subset $\{x_n : n \in \N\}$ of a metric compact set $X$ (for us $X$ stands either for $\br{0,\infty}$, $\Mn$ or $\Rn$.). 
Fix indeed a free ultrafilter $\ufrak$ on $\N$. Then, for every sequence $(x_n) \in X^{\N}$, there is a unique $x=\lim_\ufrak x_n \in X$ which satisfies:
$$\forall \epsilon>0,\exists I \in \ufrak,\, \forall n \in I,\,(d(x_n,x)<\epsilon).$$
Conversely, if $x$ is the limit of some subsequence of $(x_n)$, then we can find a free ultrafilter $\ufrak$ on $\N$ such that $\lim_{\ufrak}x_n=x$. This is a consequence of the following fact. For every infinite subset $I \subset \N$, there is a free ultrafilter on $\N$ containing $I$ \cite[Corollary 3.5 and proof of Lemma 3.8]{BeSl69}.

\subsection{Universal theory} \label{SubSecUniv}

A universal sentence is a formula of type $$\forall x_1 \, \forall x_2 \, \dots\, \forall x_n \, \varphi(x_1,\dots,x_n)$$ where $\varphi$ is a quantifier-free formula built up from the usual logic symbols and group operations using every variable $x_1,\dots,x_n$ \cite{CK90}. Note for instance that a group is metabelian if and only if it satisfies the universal sentence $$\forall x \forall y \forall z \forall t (\br{\br{x,y},\br{z,t}}=1)$$ where $\br{x,y}=xyx^{-1}y^{-1}$.
 The universal theory $\Thu(G)$ of a group $G$ is the set of universal sentences satisfied by $G$. The existential theory $\The(G)$ is defined similarly. 
Since an existential sentence is the negation of a universal sentence, the universal theories of two groups coincide if and only if their existential theories coincide. For any universal sentence $\varphi$, the set of marked groups of $\Mn$ in which $\varphi$ is true is a closed subset \cite[Proposition 5.2]{ChGu05}. In particular, the subset of all metabelian groups of $\Mn$ is compact. It is also countable as it identifies with the set normal subgroups of the free metabelian group on $n$ generators and the latter group satisfies the maximal condition on normal subgroups \cite{Hal54}. Thus the closure of $\Cu$ in $\M$ is a countable compact set consisting of metabelian groups only.
\subsection{Group presentations} \label{SubSecGP}
Let $G$ be a metabelian group generated by $S=(g_1,\dots,g_n)$ and let $R \subset F_n$. We say that $\PresM{x_1,\dots,x_n}{r=1: r \in R}$ is a \emph{presentation of $G$ in the variety of metabelian groups} if the kernel of the homomorphism $\pi_S: F_n \longrightarrow G$ induced by $x_i \mapsto g_i$ is the normal closure of $R \cup F_n''$ in $F_n$. By \cite{Hal54}, any finitely generated metabelian group $G$ can presented in this way using a finite set $R$. In order to present groups in $\Cu$ and their limits, we use the following notations.

Let $F$ be the free group of basis $(a,b)$ and set $c=\br{a,b}=aba^{-1}b^{-1}$.
For $P(x)=\sum_i \alpha_ix^i \in \Zx, Q(x,y)=\sum_{i,j}\beta_{ij}x^iy^j \in \Zxy$, define $$b^{P(a)}=\prod_{i}a^ib^{\alpha_i}a^{-i},\quad c^{Q(a,b)}= \prod_{i,j}(a^ib^j)c^{\beta_{ij}}(a^ib^j)^{-1}$$ where the product is taken with respect to the natural order of $\Z$ and the corresponding lexicographical order of $\Z^2$.

Let $\lambda \in \C \setminus \{0,1\}$ be an algebraic number over $\Z$ and let $P_{\lambda} \in \Z \br{x}$ be its minimal polynomial over $\Z$. With the notations of Theorem \ref{ThA}, we have 
$$\begin{array}{l} M=\PresM{a,b}{\emptyset},\quad \Z_p \wr \Z=\PresM{a,b}{c^{b-1}=b^p=1}, \\ 
M(\lambda)=\left\{
\begin{array}{l} \PresM{a,b}{c^{y-1}=c^{\Phi_m(a)}=1} \text{ if } \lambda \text{ is a root of unity of order } m \ge 2, \\
 M(\lambda)=\PresM{a,b}{c^{P_{\lambda}(a)}=1} \text{ else,} \end{array}\right. 
\\G(\lambda)=\PresM{a,b}{c^{b-1}=b^{P_{\lambda}(a)}=1}.\end{array}$$
More generally, define for every ideal $\aid \subset \Zxy$ the marked groups
$$\Ma=\PresM{a,b}{c^{P(a,b)}=1,\, P \in \aid}$$

 For every ideal $\bid \subset \Zx$, define also
$$ G(\bid)=\PresM{a,b}{c^{b-1}=b^{P(a)}=1,\, P \in \bid}$$
The groups $M(\aid)$ and $G(\bid)$ are marked by the pair $(a,b)$ and this marking will be omitted.
With these notations, we have $$M(y-1)=G(0)=\Z \wr \Z ,\,G(p)=\Z_p \wr \Z \text{ and } M(\Zxy)=G(x-1)=\Z^2.$$

In general, the groups $M(\aid)$ and $G(\bid)$ are not finitely presented in the usual sense. For instance, $G(\lambda)$ is finitely presented if and only if $\lambda$ or $\lambda^{-1}$ is an algebraic integer \cite[Theorem C]{BS78}. Using the Bieri-Strebel invariant defined in \cite{BiSt80}, it is easy to show that $M(\lambda)$ is finitely presented if and only if $\lambda$ is a root of unity.

\begin{remark} \label{RemBS}
Note that $G(n)=BS(1,n)$. Moreover, the group $G(m/n)$ can be obtained as a limit in $\M$ of marked groups isomorphic to $$BS(m,n)=\Pres{a,b}{ab^ma^{-1}=b^n}\, (m,n \in \Z\setminus\{0\})$$ \cite[Subsection 1.8]{BS76}. 
\end{remark}

\section{2-generated limits} \label{SecLimits}

In this section, we determine the set of all possible limits in $\M$ of two-generated abelian-by-(infinite cyclic) groups $G=R \rtimes \Z$.

\subsection{Convergence of groups and ideals} \label{SubSecFrom}

We identify $\Ru$ (resp. $\Rd$) with the set of ideals of $\Zx$ (resp. $\Zxy$) endowed with the patch topology.
We first state some elementary properties of the maps $M$ and $G$ defined in Section \ref{SubSecGP}.
\begin{lemma} \label{LemDisc}
\begin{itemize}
\item[$(i)$] The maps $\aid \mapsto \Ma$ and $\bid \mapsto \Gb$ are injective.
\item[$(ii)$] The map $M:\Rd \longrightarrow \mathcal{M}_2$ is continuous. Moreover $$M(\aid) \cdot \psi=M(\psi^{-1} \cdot \aid)$$ for every ideal $\aid \subset \Zxy$ and every $\psi \in \Aut(F_2)$.
\item[$(iii)$] $\br{\Ma}$ and $\br{\Gb}$ are discrete subsets of $\M$. If moreover $\bid$ is prime and $\bid \notin \{0,(x-1)\}$, then these subsets are disjoint.
\end{itemize}
\end{lemma}
\begin{proof}
Both $(i)$ and $(ii)$ follow from the definitions.

$(iii)$: Let $G \in \{\Ma,\Gb\}$. In order to prove the first part, it suffices to show that $G$ is isolated in $\br{G}$ by \cite[Lemma 1]{CGP07}; it actually holds for any finitely generated metabelian group $G$. As a finitely generated metabelian groups satisfies the maximal condition on normal subgroups \cite{Hal54}, 
we have: $(1)$ $G$ is finitely presented in the variety of metabelian groups, $(2)$ $G$ enjoys the Hopf property. By $(1)$, there is a neighborhood $V$ of $G$ such that the map $(a,b) \mapsto S$ induces an epimorphism from $G$ onto $H$ for every $(H,S) \in V \cap \br{G}$. It follows from $(2)$ that $(V\setminus \{G\}) \cap \br{G}=\emptyset$.

Let us show the remaining claim. Notice first that $M(\aid)=G(\bid)$ holds if and only if there is an ideal $\bid'$ of $\Zx$ such that $\aid=(\bid',y-1)$ and $\bid=(x-1)\bid'$. Indeed, the equality $M(\aid) = G(\bid)$ implies that $y-1 \in \aid$. Thus there exists $\bid' \subset \Zx$ such that $\aid=(\bid',y-1)$. Then we easily check that $M(y-1,\bid')=G((x-1)\bid')$ and deduce from $(i)$ that $\bid=(x-1)\bid'$. As $\Aut(F_2)$ acts transitively on both $\br{M(\aid)}$ and $\br{G(\bid)}$ by Proposition \ref{PropTsys}, the set $\br{M(\aid)} \cap \br{G(\bid)}$ is non-empty if and only if there is $\psi \in \Aut(F_2)$ such that $M(\psi^{-1} \cdot \aid)=G(\bid)$. The result follows from the above remark.
\end{proof}

Let $G=R \rtimes \Z$ be a two-generated abelian-by-(infinite cyclic) group. Keeping the notations of Subection \ref{SubSecGP}, $R$ is generated as a ring by $\{\lambda^{\pm1}\}$, so we write $R=\Z_q \br{\lambda^{\pm1}}$ where $q$ is the characteristic of $R$. Let $\bid_{R}$ be the kernel of the ring homomorphism from $\Zx$ onto $R$ induced by $x \mapsto \lambda$. We have then $$G=G(\bid_R).$$ 
We fix two sequences $(G_n)$ and $(S_n)$ where $G_n=R_n \rtimes \Z$ and $S_n$ is a generating pair of $G_n$. Set $$R_n=\Z\br{\lambda_n^{\pm 1}},\, (a_n,b_n)=S_n,\,(k_n,l_n)=\sigma(S_n)$$ Let $\mathfrak{U}$ be an ultra-filter over $\N$ and let $\prod_{\mathfrak{U}}R_n$ be the ultra-product of the rings $R_n$ with respect to $\mathfrak{U}$, i.e. the quotient ring of the cartesian product $\prod R_n$ by the following equivalence relation: the sequences $(r_n)$ and $(r_n')$ are equivalent if $\{n: r_n=r'_n\}\in \ufrak$. We denote by $\prod_{\text{Fr\'echet}}R_n$ the quotient ring of $\prod R_n$ by the following equivalence relation: the sequences $(r_n)$ and $(r_n')$ are equivalent if $\{n: r_n\neq r'_n\}$ is finite.
We denote by $\aid_{\mathfrak{U}}$ (resp. $\aid$) the kernel of the ring homomorphism from $\Zxy$ to $\prod_{\mathfrak{U}}R_n$ (resp. $\prod_{\text{Fr\'echet}}R_n$) induced by $$x \mapsto (\lambda_n^{k_n}),\, y \mapsto (\lambda_n^{l_n})$$
We denote by $\aid_n$ the kernel of the ring homomorphism from $\Zxy$ to $R_n$ induced by $x \mapsto \lambda_n^{k_n},\, y \mapsto\lambda_n^{l_n}$.
Clearly, we have \begin{equation} \label{EqLimAn}\aid_{\ufrak}=\lim_{\ufrak}\aid_n\end{equation} with respect to the patch topology of $\Rd$. If the sequence $(\aid_n)$ converges, then it converges to $\aid$ and the identity map of $\prod R_n$ induces a ring isomorphism from $\prod_{\text{Fr\'echet}}R_n$ to $\prod_{\mathfrak{U}}R_n$, i.e. $\au=\aid$.

\begin{lemma} \label{LemConvAu}
\begin{itemize}
\item[$(i)$]
Assume that $|k_n|+|l_n|$ tends to infinity. Then $$\lim_{\mathfrak{U}}(G_n,S_n)=M(\aid_{\mathfrak{U}}).$$

\item[$(ii)$] 
Assume that $(k_n,l_n)$ is stationary. The sequence $(G_n,S_n)$ converges if and only if the sequence $(\bid_{R_n})$ converges in $\Ru$. In this case, its limit $G$ is abstractly isomorphic to $G(\bid)$ with $\bid=\lim \bid_{R_n}$.
\item[$(iii)$] If $(k_n,l_n)$ is stationary and $(\lambda_n) \in (K_p \setminus \{0\})^{\N}$ has no constant subsequence, then $\lim (G_n,S_n)=G(p)$.
\end{itemize}
\end{lemma}

Assertion $(ii)$ merely means that the map $\bid \mapsto G(\bid)$ is an homeomorphism from $\mathcal{R}_1$ onto its image in $\mathcal{M}_2$. 
We shorten notations denoting by $b^r$ the group element $(r,0) \in R \rtimes \Z$.
\begin{proof}
$(i)$: Let $w \in F(a,b)$. There are $u,v \in \Z$ and $W(x,y) \in \Zxy$ such that $w=c^{W(x,y)}b^va^u$ holds in $M$. For every $n$ the following holds in
$G_n$: $$w=b^{r_n}a^{uk_n+vl_n}\text{ with } r_n=\Delta(S_n)W(\lambda_n^{k_n},\lambda_n^{l_n}) \text{ and } \Delta \text{ as in Subsection \ref{SubSecGP}}.$$

 Assume that $w=1$ in $\Mau$. It follows that $u=0=v$ and $W=0$. Therefore $r_n=0$ holds in $R_n$ for every $n$ and hence $w=1$ in $G_n$ for every $n$. 
Thus $w=1$ holds in the limit group $G$.

 Assume that $w=1$ in the limit $G$. There is a subset $I \in \ufrak$ such that $w=1$ holds in $G_n$ for every $n \in I$. It follows that 
 $uk_n+vl_n=0$ for infinitely many $n$. As $k_n$ and $l_n$ are coprime and $|k_n|+|l_n|$ tends to infinity, we deduce that $u=0=v$. 
Since $\Delta(S_n)$ is a unit of $R_n$ by Proposition \ref{PropDelta}.$i$, we deduce that $W(\lambda_n^{k_n},\lambda_n^{l_n})=0$ holds in $R_n$ for every $n \in I$, 
which proves that $W(x,y) \in \au$ and hence $w=1$ holds in $\Mau$.

$(ii)$: Replacing $S_n$ by $S_n \cdot \psi$ for some Nielsen tranformation $\psi$, we can assume that $S_n=(a,b)$ for all $n$. 
Let $w \in F(a,b)$. There are $u\in \Z$ and $W(x) \in \Zx$ such that $w=b^{W(x)}a^u$ holds in $M(y-1)$. For every $n$ we have
$w=b^{W(\lambda_n)}a^{u}$ in $G_n$. We obtain our condition of convergence reasonning as above.

$(iii)$: If $W \in \Zx$ belongs to infinitely many $\bid_{R_n}$ then $W$ has infinitely many zeros in $K_p$, which proves that $\lim \bid_{R_n}=(p)=\bid$. Therefore $G=G(p)$ by $(ii)$.
\end{proof}

\subsection{A convergence criterion based on absolute values} \label{SubSecCrit}

We now give several conditions on $(\lambda_n) \in K_p \setminus \{0\}$ and $(k_n,l_n)$ under which the convergence $$\lim (G_n,S_n)=M(p)$$ holds in $\M$.
 If $|k_n|+|l_n|$ tends to infinity, this means that $$\aid_{\ufrak}=(p)$$ holds for every free ultrafilter $\ufrak$ 
over $\N$, by Lemma \ref{LemConvAu}. In other words, the ring elements $x_{\ufrak}=(\lambda_n^{k_n})$ and $y_{\ufrak}=(\lambda_n^{l_n})$ are algebraically free
 elements of $R_{\ufrak}$ over $\Z_p$. Thus we are to establish a criterion of algebraic freeness (Lemma \ref{LemFi}).
We introduce new definitions in order to state this freeness criterion. 
Our convergence criterion (Proposition \ref{PropCrit}) and its proof are postponed at the end of this subsection.

Let $R$ be a ring and let $A$ be a subring of $R$. A family $\F \subset R$ is said to be \emph{free over $A$} if $\F$ is a basis of a free $A$-submodule of $R$. 

\begin{remark}
If $\F$ is free over $A$ and $\F$ is a free abelian subsemigroup (resp. subgroup) of the multiplicative semigroup $R$ (resp. $R^{\times}$), then its basis elements are algebraically free over $A$.
\end{remark}
Assume now that $R=\prod R_n$ is a countable cartesian product of arbitrary rings $R_n$. We consider the following stronger version of freeness:
\begin{definition}
 Let $\mathcal{F}\subset R$ and let $M \subset R$ be the $A$-submodule generated by $\mathcal{F}$. 
The family $\mathcal{F}$ (resp. the $A$-module $M$) is said to be \emph{subsequence-free over $A$} 
if for every infinite subset $X \subset \N$ the projection of $\mathcal{F}$ to $\prod_{n \in X} R_n$ is free over $A$. 
\end{definition}

\begin{remark}
Let $R_n=A=\Z$. Let $(f_n) \in \Z^{\N}$ be a non-constant sequence. Then $(f_n)$ belongs to some subsequence-free $\Z$-submodule of $\Z^{\N}$ if and only if $|f_n|$ tends to infinity.
\end{remark}

We assume moreover that for every $n$, the ring $R_n$ has an abolute value denoted by $|\cdot|_n$. In particular, each $R_n$ is an integral domain.

\begin{lemma} \label{LemAbs}
Let $\mathcal{F}$ be a subgroup of $R^{\times}$. Let $M \subset R$ be the $A$-submodule generated by $\mathcal{F}\setminus \{1\}$. 
We assume that the two following hold:
\begin{itemize}
\item
for every $a \in A \setminus \{0\}$, $|a(n)|_n$ is bounded away from $\{0,\infty\}$,
\item for every $f$ in $M$ and every free ultrafilter $\ufrak$ on $\N$, we have $$\lim_{\ufrak}|f(n)|_n\in\{0,\infty\}.$$
\end{itemize}
Then $\mathcal{F}$ is subsequence-free on $A$.
\end{lemma}
\begin{proof}
Assume that there are $a_1,\dots,a_m$ in $A$, $f_1,\dots,f_m$ 
pairwise distinct in $\mathcal{F}$, and an infinite subset $I \subset \N$ 
such that $\sum_{i=1}^m a_i(n) f_i(n)=0$ holds for every $n \in I$.
For every $j=1,\dots,m$ and for all $n \in I$, we have 
$$\sum_{i\neq j} a_i(n) (f_if_j^{-1})(n)=-a_j(n).$$ 
Let $\ufrak$ be an ultra-filter on $\N$ containing $I$.
Since $\lim_{\ufrak} |\sum_{j}^m a_i(n) f_i(n)|_n \in \{0,\infty\}$, we deduce that $a_j=0$ for every $j$. Hence $\mathcal{F}$ is subsequence-free on $A$.
\end{proof}

Now we are able to state our main algebraic freeness criterion.
\begin{lemma} \label{LemFi}
Let $(\lambda_n) \in R^{\times}$. Assume that the two following hold:
\begin{itemize}
\item for every $a \in A \setminus \{0\}$, $|a(n)|_n$ is bounded away from $\{0,\infty\}$,
\item $|\lambda_n|_n>1$ is bounded away from $1$.
\end{itemize}
Let $\{f_i\}_i \subset \Z^{\N}$ be the basis of a  subsequence-free module over $\Z$ without non-zero constant sequences. Then $\{(\lambda_n^{f_i(n)})\}_i$ freely generates an abelian free subgroup $\F \subset R^{\times}$ which is  subsequence-free over $A$.
\end{lemma}

\begin{sproof}{Proof of Lemma \ref{LemFi}}
It follows directly from our assumption on $\{f_i\}_i$ that the abelian group $\mathcal{F}$ is freely generated by the sequences $(\lambda_n^{f_i(n)})$. So it suffices to prove that $\mathcal{F}$ satisfies the hypothesis of 
Lemma \ref{LemAbs}. Let $\ufrak$ be a free ultra-filter on $\N$.
Let $a_1,\dots,a_m \in A,\, g_1,\dots,g_m \in \bigoplus_i \Z f_i$ and set $F(n)=\sum_{i=1}^m a_i(n){\lambda_n}^{g_i(n)}$. We will prove that $\lim_{\ufrak}|F(n)|_n \in \{0,\infty\}$ by induction on $m$. If $m=1$, the result is obvious. Assume that $m \ge 2$.
We can suppose that $a_1 \neq 0$ and $\lim_{n \to \infty}g_1(n)=\infty$.
 Write $F(n)={\lambda_n}^{g_1(n)}(a_1(n)+F_1(n))$.
 By induction hypothesis, we have $\lim_{\ufrak}|F_1(n)|_n \in \{0,\infty\}$. If follows that $\lim_{\ufrak}|F(n)|_n=\infty$, which completes the proof.
\end{sproof}

We still need further definitions and a last elementary freeness result.
Let $f,a \in R$. We say that \emph{$f$ dominates $a$} if $\{n \in \N: f(n)=0\}$ is finite and $$\lim_{n \to \infty} \frac{|a(n)|_n}{|f(n)|_n}=0.$$
We say that $\mathcal{F} \subset R$ dominates $B \subset R$ if every element of $\mathcal{F}$ dominates every element of $B$.

\begin{lemma} \label{LemDom}
Let $I$ be a totally ordered set and let $\F=\{f_i: i \in I\} \subset R^{\times}$. Assume that 
\begin{itemize}
\item $A\setminus \{0\}$ dominates $A \mathcal{F}^{-1}$,
\item $f_if_j^{-1} \in \mathcal{F}$ whenever $i>j$.
\end{itemize} Then $\F$ is free over $A$. 
\end{lemma}

\begin{remark}
 If  $\F$ is of the form $\{f^i: i \in \N \setminus \{0\}\}$ for some $f \in R^{\times}$
 and satisfies the hypothesis of Lemma \ref{LemFi}, then $f$ is transcendent over $A$.
\end{remark}

\begin{proof}
Let $a_1,\dots,a_m \in A$ and $f_{i_1},\dots,f_{i_m} \in \mathcal{F}$ such that $i_1>\dots>i_m$ and $\sum_{j=1}^m a_j f_{i_j}=0$. We show be induction on $m$ that $a_j=0$ for $j=1,\dots,m$. The case $m=1$ is immediate. Assume that $m \ge 2$. If $a_1 \neq 0$, the following identity would hold in the fraction field of $R_n$ $$1=F(n) \text{ with } F(n)=-\sum_{j=2}^m a_1(n)^{-1}(a_j(n)f_{i_j}(n)f_{i_1}(n)^{-1})$$ for all $n$ large enough. As $|F(n)|_n$ tends to zero by hypothesis, this would lead to $1=0$, a contradiction. Hence $a_1=0$ and the result follows by induction hypothesis. 
\end{proof}

To state our convergence criterion, we consider three \emph{natural conditions} on the sequence $(k_n,l_n)$:

\begin{itemize}
\item[$(1)$] $(k_n,l_n)$ is stationary;
\item[$(2)$] $ \vert k_n\vert +\vert l_n \vert$ tends to infinity
and there is $\nu=(\alpha,\beta)\in \Z^2 \setminus \{(0,0)\}$ and
$\gamma\in \Z$ such that: $$\alpha k_n+\beta l_n=\gamma \text{ for
all } n \text{ large enough}.$$
\item[$(3)$]$\vert k_n\vert +\vert l_n \vert$ tends to infinity and for all $(\alpha,\beta) \in \Z^2 \setminus \{(0,0)\}$ and all
$\gamma\in \Z$ we have: $$\alpha k_n+\beta l_n \neq \gamma \text{ for
all } n \text{ large enough}.$$
\end{itemize}
Condition $(3)$ means that $\{(k_n),\,(l_n)\}$ is the basis of a subsequence-free module over $\Z$ without non-zero constant sequences. Note that if $(3)$ does not hold, then there is a subsequence of $(k_n,l_n)$ which satisfies $(2)$. Note also that if $(1)$ is satisfied by no subsequence of $(k_n,l_n)$ then $ \vert k_n\vert +\vert l_n \vert$ tends to infinity.

We are in position to prove the main result of this subsection.
\begin{proposition} \label{PropCrit}
 Assume that each $R_n$ is a subring of an integral domain $D$ with characteristic subring $\Z_p$. Let $(\lambda_n) \in R^{\times}$ be
without constant subsequence. Assume moreover that $|\lambda_n|_n$ is bounded away from $1$ and
 $(k_n,l_n)$ satisfies either $(2)$ or $(3)$. Then $\lim (G_n,S_n)=M(p)$.
\end{proposition}

\begin{proof}
By Lemma \ref{LemConvAu}, it suffices to show that $\aid_{\ufrak}=(p)$ holds for every free ultrafilter $\ufrak$ on $\N$.
For such a $\ufrak$, consider $W \in \au$. By definition, $W(\lambda_n^{k_n},\lambda_n^{l_n})=0$ for every $n$ in $I$ for some $I \in \ufrak$. Set $R_n=\Z_p \br{{\lambda_n}^{\pm 1}},\,R=\prod_{n \in I} R_n$.

Assume that $(2)$ holds. Replacing $W$ by $W \cdot \psi$ for some Nielsen transformation $\psi$, we can assume that 
$k_n=\gamma \in \Z \setminus \{0\},\,l_n>0$ for all $n \in I$ and $\lim_{\ufrak} |\lambda_n|_n>1$. Thus $W(\lambda_n^{\gamma},\lambda_n^{l_n})=0\text{ holds in }D$ for all $n  \in I$. Since $A=\Z_p \br{(\lambda_n^{\gamma})_{n \in I},(\lambda_n^{-\gamma})_{n \in I}}$ dominates $A \mathcal{F}^{-1}$ with $\F=\{(\lambda_n^{il_n})_{n \in I}: i \in \N \setminus \{0\}\}$, it follows from Lemma \ref{LemDom} that $(\lambda_n^{l_n})_{n \in I}$ is transcendent over $A$. By hypothesis, $(\lambda_n^{\gamma})_{n \in I}$ is transcendant over $\Z_p$. Consequently, $W \in (p)$.

Assume that $(3)$ holds. Let $\F \subset R^{\times}$ be the subgroup generated by 
$$\{(\lambda_n^{k_n})_{n \in I},(\lambda_n^{l_n})_{n \in I}\}.$$ Then $\F$ is free over $\Z_p$ by Lemma \ref{LemFi}. Hence $W \in (p)$.
\end{proof}

For a constant sequence $(\lambda_n) \in (K \setminus \{0\})^{\N}$ with $K$ an arbitrary field, Proposition \ref{PropCrit} applies in the following case:

\begin{lemma} \label{LemNonRoot} \cite[Th.8, p77]{Wei95}
Let $K$ be any field and assume that $\lambda \in K$ is not a root of unity. Then there exists an absolute value $|\cdot|$ defined on the subfield of $K$ generated by $\lambda$ such that $|\lambda| \neq 1$.
\end{lemma}

\subsection{The case of a single element} \label{SubSecSingle}

\paragraph{Roots of unity}

\begin{proposition} \label{PropRoot} Let $\lambda \in K_p\setminus \{1\}$ be a primitive $m$-th root of unity. Let $(S_n)$ be a sequence
of generating pairs of $\G$ and let $(k_n,l_n)=\sigma(S_n)$. We denote by $\ok_n$ and $\ol_n$ the images of $k_n$ and $l_n$ in $\Z/m\Z$. 
\begin{itemize}
\item[$(i)$] The sequence $(\G,S_n)$ converges
if and only if one of the following conditions holds:

\begin{itemize}
\item[$(1)$] $(|k_n|,|l_n|)$ is stationary;
\item[$(2)$] $ \vert k_n\vert +\vert l_n \vert$ tends to infinity and there is $(\ok,\ol) \in (\Z/m\Z)^2$ such that $(\ok_n,\ol_n)=(\ok,\ol)$
 for all $n$ large enough.
\end{itemize}
\item[$(ii)$]
 If $(\G,S_n)$ satisfies $(2)$ then it converges to $M(\aid) \cdot \psi_{\ok,\ol}$ with $$\aid=(\Phi_m(x),y-1,p)$$ 
and $\psi_{\ok,\ol}$ is a Nielsen transformation statisfying $(\ok,\ol) \cdot \psi_{\ok,\ol}=(\overline{1},\overline{0})$.
\end{itemize} 
\end{proposition}

\begin{proof}
To begin with, we prove that the sequence $(G(\lambda),S_n)$
converges to the limit prescribed by $(ii)$ if $(2)$ holds. Replacing $S_n$ by $S_n \cdot \psi_{\ok,\ol}$, we can assume that $(\ok,\ol) =(\overline{1},\overline{0})$. It suffices to show that $\au=(\Phi_m(x),y-1)$ for any free ultrafilter $\ufrak$ on $\N$ by Lemma \ref{LemConvAu}.
Let $W \in \Zxy$. The polynomial $W$ belongs to $\au$ if and only if  $W(\lambda,1)=0$ holds in $K_p$, i.e. $W \in (\Phi_m(x),y-1,p)$. Thus condition $(2)$ is a sufficient condition of convergence, and clearly so is $(1)$.

Consider now a converging sequence $(\G,S_n)$. If $(|k_n|,|l_n|)$ has a constant subsequence, then $(\G,S_n)$ must converge to $(\G,S)$ for some generating pair $S$ of $\G$ by Proposition \ref{PropDelta}.$iii$ and Remark \ref{RemRec}. Else $(S_n)$ has a subsequence which satisfies $(2)$. It follows from Lemma \ref{LemDisc}.$iii$ that $(S_n)$ must satisfy either $(1)$ or $(2)$.
\end{proof}

\paragraph{Infinite order case}

\begin{theorem} \label{ThConv} Let $\lambda \in \C \setminus \{0\}$ be an
an algebraic number over the rationals which is not a root of unity. Let $(S_n)$ be a sequence
of generating pairs of $\G$ and let $(k_n,l_n)=\sigma(S_n)$.
\begin{itemize}
\item[$(i)$] The sequence $(\G,S_n)$ converges
if and only if one of the following conditions holds:

\begin{itemize}
\item[$(1)$] There is a sequence $(\epsilon_n) \in \{\pm 1\}^{\N}$ such that $(\epsilon_n k_n, \epsilon_n l_n)$ is stationary, $(\epsilon_n)$ being moreover stationary when $\lambda$ is not reciprocal;
\item[$(2)$] $ \vert k_n\vert +\vert l_n \vert$ tends to infinity
and there is $(\alpha,\beta)\in \Z^2 \setminus \{(0,0)\}$ and
$\gamma\in \Z$ such that: $$\alpha k_n+\beta l_n=\gamma \text{ for
all } n \text{ large enough}.$$
\item[$(3)$]$\vert k_n\vert +\vert l_n \vert$ tends to infinity and for all $(\alpha,\beta) \in \Z^2 \setminus \{(0,0)\}$ and all
$\gamma\in \Z$ we have: $$\alpha k_n+\beta l_n \neq \gamma \text{ for
all } n \text{ large enough}.$$
\end{itemize}
\item[$(ii)$]
\begin{itemize} \item[$\bullet$] If $(2)$ holds, with the normalization $\gamma>0$ and $\gcd(\alpha,\beta)=1$, then $(\Gl,S_n)$ converges
to $M(P_{\lambda^{\gamma}}(x^{\alpha}y^{\beta}))$ where $P_{\lambda^{\gamma}} \in \Z \br{x}$ is the minimal polynomial of $\lambda^{\gamma}$ over $\Z$.
\item[$\bullet$] If $(3)$ holds then $(G(\lambda),S_n)$ converges to $M$.
\end{itemize}
\item[$(iii)$] Let $\lambda \in K_p$ be transcendant over $\Z_p$. Then $(\G,S_n)$
 converges to $M(p)$ if and only if $|k_n|+|l_n|$ tends to infinity.
\end{itemize}
\end{theorem}
Let  $M(\alpha,\beta,\gamma)=M(P_{\lambda^{\gamma}}(x^{\alpha}y^{\beta}))$. Note that $M(\alpha,\beta,\gamma)=M(P_{\lambda^{\gamma}}(x)) \cdot \psi_{\alpha,\beta}$ for some Nielsen transformation $\psi_{\alpha,\beta}$ depending only on $(\alpha,\beta)$.

\begin{sproof}{Proof of Theorem \ref{ThConv}}
First we prove that the sequence $(\Gl,S_n)$
converges to the limit prescribed by $(ii)$ if one of the
conditions $(1)$, $(2)$ or $(3)$ holds.

If $(1)$ holds, it is a consequence of the remark following Proposition \ref{PropCrit}.

\it Condition $(2)$\rm: There are $(\alpha,\beta)\neq (0,0)\,,\gamma>0$ such that $\gcd(\alpha,\beta)=1$ and $\alpha
k_n+\beta l_n=\gamma$ for $n$ large enough. Set
$$G_n=(\G,S_n),\, \aid=(P_{\lambda^{\gamma}}(x^{\alpha}y^{\beta}))$$ 
Possibly replacing $G_n$ by $G_n \cdot \psi$ and $\Ma$ by $\Ma \cdot \psi$ 
for some Nielsen transformation $\psi$, we can assume that $\aid=(P_{\lambda^{\gamma}}(x))$ and $k_n=\gamma$ for $n$ large enough.
Let $\ufrak$ be a free ultra-filter on $\N$. It suffices to show that $\au=(P_{\lambda^{\gamma}}(x))$ by Lemma \ref{LemConvAu}.
Let $W \in \Zxy$. The polynomial $W$ belongs to $\au$ if and only if there is $I \in \ufrak$ such that $W(\lambda^{\gamma},\lambda^{l_n})=0$ holds in $\C$ for every $n \in I$. As $\lambda$ is not a root of unity, this is equivalent to $W \in (P_{\lambda^{\gamma}}(x))$.

\it{Condition $(3)$}\rm: By Lemma \ref{LemNonRoot}, there is some absolute value defined on $\Z\br{\lambda^{\pm 1}} \subset \C$ such that $|\lambda| \neq 1$. It follows from Lemma \ref{LemConvAu} and Proposition \ref{PropCrit} that $(G_n)$ converges to $M$.

 Now, we prove that one of the conditions $(1)$, $(2)$
and $(3)$ is necessary, provided that $G_n=(G(\lambda),S_n)$
converges to $(G,S)$. Let us consider two cases.

\it{ Case 1}\rm: $(G,S)=(\G,S)$. As $(\G,S)$ is
isolated in $\br{G}$ by Lemma \ref{LemDisc}, the sequence $(G_n)$ is stationary and hence
$(1)$ holds by Proposition \ref{PropDelta}.$iii$ and \ref{RemRec}.

\it{ Case 2}\rm: $(G,S) \neq (\G,S)$. For any subsequence of $(G_n)$, the condition $(1)$ does not
hold, so that $ \vert k_n\vert +\vert l_n \vert$ tends to infinity.
 We reason by contradiction, assuming
that neither $(2)$ nor $(3)$ holds. Since $(3)$ does not hold, there
is a subsequence of $(G_n)$ for which $(2)$ holds. Hence we have
$G=M(\alpha,\beta,\gamma)$ for some $(\alpha,\beta,\gamma)
\in \Z^3$, with $\gcd(\alpha,\beta)=1$ and $\gamma>0$. The set $I=\{n
\in \N \, \vert \, \alpha k_n+\beta k_n=\gamma \}$ is then infinite.
As $(2)$ does not hold, there is an increasing sequence $(n_k)$ with $n_k \in \N \setminus I$. As $M \neq
M(\alpha,\beta,\gamma)$, condition $(3)$ does not hold for $(G_{n_k})$. Thus, we can extract from it another
subsequence $(G_{n_{k_l}})$ for which $(2)$ holds. This subsequence
converges to some $M(\alpha',\beta',\gamma')$ with $\gcd(\alpha',\beta')=1$ and $\gamma'>0$, which must coincide with $M(\alpha,\beta,\gamma)$. It follows easily from Lemma \ref{LemDisc} that $(\alpha',\beta',\gamma')=(\alpha,\beta,\gamma)$. Consequently, we have
$n_{k_l} \in I$ for all $l$ large enough. This contradicts the fact that $n_{k_l}\in \N \setminus I$ for all $l$.

Proof of $(iii)$: The very same arguments used in $(i)$ and $(ii)$ shows that $(1),\,(2)$ and $(3)$ 
are sufficient conditions of convergence and that the limit is $M(p)$ in every case. 
As any subsequence of $(\G,S_n)$ has a subsequence which satisfies  one these conditions, the result follows.
\end{sproof}

\begin{remark}\label{RemTop}
Proposition \ref{PropRoot} and Theorem \ref{ThLimMa} enable us to depict $E=\overline{\br{\Gl}}$ as a topological space on which $\Aut(F_2)$ acts continuously. We see indeed that $E_0=\br{\Gl}$ is the set of isolated points in $E$, i.e. the points of Cantor-Bendixson rank $0$. 

The set $E_0$ accumulates on $E_1=\br{M(\Phi_m(x),y-1)}$ if $\lambda$ is a root of unity distinct from $1$, else $E_0$ accumulate on $E_1=\bigcup_{k>0} \br{M(P_{\lambda^k})}$, $E_1$ being the set of points with Cantor-Bendixson rank $1$ in each case. In the latter case $E_1$ accumulates in turn on $E_2=\{M\}$, the unique point of Cantor-Bendixson rank $2$.
Thus, if $\lambda$ is a $m$-th root of unity $(m \ge 2)$ then $E$ is homeomorphic to the finite disjoint union $\bigsqcup_{(\ok,\ol) \in \V(\Z/m\Z,2)}\overline{\N}$ where $\overline{\N}=\N \cup \{\infty\}$ is the one-point compactification of the discrete space $\N$.
Else $E$ is homeomorphic to $\overline{\N}^2$ and $M$ identifies with $(\infty,\infty)$ while the isomorphism classes $\br{M(P_{\lambda^k})}$ identify with some infinite subsets partitionning $(\{\infty\} \times \N) \cup (\N \times \{\infty\}) \setminus \{(\infty,\infty)\}$.

The closure of $\br{\Gl : \lambda \in X}$ enjoys a similar description for algebraically tame subsets $X \subset \C \setminus \{0\}$ (Corollary \ref{CorBounded}).
\end{remark}

\subsection{The closure of $\Cp$} \label{SubSecClosure}
Recall that $\Cp$ is the class of two-generated groups of the form $R \rtimes \Z$ with $R$ an integral domain of characteristic $p$.
For $\lambda \in K_p \setminus \{0\}$, we set $$G(\lambda)=\Z_p\br{\lambda^{\pm 1}} \rtimes \Z$$ where $\Z$ acts by multiplication of $\lambda$ on $\Z_p\br{\lambda^{\pm 1}}$. It is marked by the \emph{canonical pair} $(a,b)$ where $a$ is the positive cyclic generator of the right factor $\Z$ and $b$ is the identity of the ring $\Z_p\br{\lambda^{\pm 1}}$.
We observe that $G(\lambda)$ is isomorphic to $G(\mu)$ as a marked group if $\lambda,\mu \in K_p \setminus \{0\}$ are either both transcendant or both algebraic and $\lambda$ is a conjugate of $\mu^{\pm1}$ (see Lemma \ref{LemThGl} for the converse). 

As in Subsection \ref{SubSecFrom}, we associate to every sequence of marked groups $(G(\lambda_n),S_n)$ with $\sigma(S_n)=(k_n,l_n)$ the \emph{corresponding sequence} $(\aid_n)$ of ideals, $\aid_n \subset \Zxy$ being the kernel of the ring homomorphism induced by the map $x \mapsto \lambda_n^{k_n},\,y \mapsto \lambda_n^{l_n}$. 
The following lemma shows that the case of stationary sequences of ideals is ruled out by Theorem \ref{ThLimMa}.

\begin{lemma} \label{LemSta}
Let $(G(\lambda_n),S_n)$ be such that $\lambda_n \in K_p \setminus \{0\}$ is algebraic over $\Z_p$ for every $n$. If the corresponding sequence $(\aid_n)$ is stationary, then there are infinitely many $\lambda_n$ which are pairwise conjugates.
\end{lemma}

The proof of Lemma \ref{LemSta} relies on:

\begin{lemma} \label{LemConj}
Let $\lambda,\mu \in K_p \setminus \{0\}$ be algebraic over $\Z_p$.
\begin{itemize}
 \item[$(i)$]
If $\lambda$ is a conjugate of $\lambda^k$ for some $k \in \Z$, then either $k=\pm 1$ or $\lambda$ is a primitive $m$-th root of unity with $\gcd(m,k)=1$.
\item[$(ii)$] Let $R_{\lambda}=\Z_p \br{\lambda^{\pm 1}},R_{\mu}=\Z_p \br{\mu^{\pm 1}},\, \Sigma_{\lambda}=(\lambda^k,\lambda^l),\Sigma_{\mu}=(\mu^m,\mu^n)$ with $\gcd(k,l)=\gcd(m,n)=1$. If $(R_{\lambda},\Sigma_{\lambda})$ is isomorphic to $(R_{\mu},\Sigma_{\mu})$ as a marked ring, then $\lambda$ is a conjugate of $\mu$ or $\mu^{- 1}$.
\end{itemize}
\end{lemma}

\begin{proof}
$(i)$: If $\lambda$ is a root of unity, then $\lambda^k$ must be a root of unity with the same order $m$. Thus $\gcd(m,k)=1$. If $\lambda$ is not a root of unity, the set $\{k^n: n \in \N\}$ must be finite as the set $\{\lambda^{k^n} : n \in \N\}$ is a finite set of conjugates of $\lambda$. Hence $k=\pm 1$.

$(ii)$: By assumption, the map $\lambda^k \mapsto \mu^m,\,\lambda^l \mapsto \mu^n$ induces a ring isomorphism $\varphi$ from $R_{\lambda}$ to $R_{\mu}$ which naturally extends to fraction fields. By Bezout's lemma, we can find $u,v \in \Z$ such that $\varphi(\lambda)=\mu^{u}$ and $\varphi^{-1}(\mu)=\lambda^v$. Thus $\lambda=\lambda^{uv}$. The result follows then from $(i)$.
\end{proof}

\begin{sproof}{Proof of Lemma \ref{LemSta}}
The condition is equivalent to say that the marked rings $(R_n,\Sigma_n)$ are eventually pairwise isomorphic, where $R_n=\Z_p \br{\lambda_n^{\pm 1}}$ and $\Sigma_n=(\lambda_n^{k_n},\lambda_n^{l_n})$. Hence the result follows from Lemma \ref{LemConj}.$ii$. 
\end{sproof}

We are eventually able to prove:

\begin{theorem} \label{ThLimMa} Let $G$ be $2$-generated group. 
\begin{itemize}
 \item[$(i)$] 
If $G$ is an accumulation point in $\M$ of groups of the form $R \rtimes \Z$, then $G$ is either isomorphic to $G(\bid)$ for some ideal $\bid$ of $\Zx$ or it is isomorphic to $M(\aid)$ for some ideal $\aid$ of $\Zxy$.
\item[$(ii)$] If $G$ is an accumulation point of $\Cp$ in $\M$, then $G$ is either isomorphic to $G(p)=\Z_p \wr \Z$ or it is isomorphic to $M(\aid)$ for some prime ideal $\aid \subset \Zxy$ of the form $$(P(x,y),\, p) \text{ or }(\Phi_m(x),\, y-1,\, p)$$ where $\Phi_m(x)$ denotes the $m$-th cyclotomic polynomial over the rationals.
\item[$(iii)$] If the orbit of $\aid$ under the action of $\Aut(F_2)$ contains a prime ideal of the form 
$
(P(x),\, p) \text{ or } (\Phi_m(x),\, y-1,\, p)
$
then $M(\aid)$ is an accumulation point of $\Cp$ in $\M$.
\item[$(iv)$] Assume that $p=0$. Then all groups $M(\aid)$, with $\aid$ as in $(ii)$, are accumulation points of $\mathcal{C}_0$ in $\M$. Moreover, we have:
\begin{enumerate}
\item The group $G$ is a limit of a sequence $(G(\lambda_n))$, where each $\lambda_n$ is a root of unity, if and only if $G$ is isomorphic to one of the groups: $ M,\,\Z \wr \Z,\,M(\Phi_m(x)),\,M(\Phi_m(x),y-1) \quad (m \ge 1)$. In this case, we call $G$ a $\Phi$-limit.
\item The group $G$ is a limit of a sequence $(G(\lambda_n))$, where each $\lambda_n$ is a reciprocal algebraic integer if and only if $G$ is either a $\Phi$-limit or it is isomorphic to $M(P(x,y))$ for some irreducible reciprocal polynomial $P$.
\end{enumerate}
\end{itemize}
\end{theorem}

\begin{sproof}{Proof of Theorem \ref{ThLimMa}}
$(i)$: it follows immediatly from Lemma \ref{LemConvAu}.

$(ii)$:
Let $(L,S)=\lim (G_n,S_n)$ be a $2$-generated accumulation point of groups $G_n \in \Cp$ and let $(k_n,l_n)=\sigma(S_n)$. If $(k_n,l_n)$ has a constant subsequence, then $G$ is either isomorphic to $G(p)$ or $M(\Phi_m(x),y-1,p)$ for some $m$ by Lemma \ref{LemConvAu}.$iii$ and Proposition \ref{PropRoot}. Thus we can assume that $|k_n|+|l_n|$ tends to infinity. As in Section \ref{SubSecFrom}, let $(\aid_n)$ be the corresponding sequence of ideals and let $\aid$ its limits. Each $\aid_n$ can be written under the form $(P_n(x),y-1,p)\cdot \psi_n$ with $\psi_n \in \Aut(F_2)$ and the reduction $\overline{P_n}$ of $P_n$ modulo $p$ is irreducible over $\Z_p$. Thus every $\aid_n$ has the same height $h\in \{2,3\}$, with $h=2$ if $p=0$ and $h=3$ else. If $(\aid_n)$ accumulates on the limit ideal $\aid$, then $\aid$ has height at most $h-1$ by \cite[Corollary 3.17]{CN03}. It means that either $\aid=(p)$ or $\aid=(p,P(x,y))$ for some polynomial $P \in \Z\br{x,y}$ whose reduction modulo $p$ is irreducible in $\Z_p\br{x,y}$. If $(\aid_n)$ is stationary then there are infinitely many $\lambda_n$ which are either trancendant or algebraic pairwise conjugates by Lemma \ref{LemSta}. It follows from Proposition \ref{PropRoot} and Theorem \ref{ThConv} that the orbit of $\aid$ under $\Aut(F_2)$ contains a prime ideal of the form $(\Phi_m(x),y-1,p),(p)$ if $p>0$ or a prime ideal of the form $(P(x))$ if $p=0$. The proof of $(ii)$ is then complete.

$(iii)$:
If $\aid=(\Phi_m(x),y-1,p)$ then $M(\aid)$ is an accumulation point of $\Cp$ by Proposition \ref{PropRoot}.$ii$ and Lemma \ref{LemDisc}. Assume now that $\aid=(P(x),p)$ is a prime ideal.
Let $(\lambda_n)$ be such that $\lambda_n$ is a root of $P(x^n)$ in $K_p$ and has unbounded degree over $\Z_p$. Let $S_n$ be a generating pair of $G(\lambda_n)$ such that $\sigma(S_n)=(n,1)$. The corresponding sequence of ideals $(\aid_n)$, with same height $h$, accumulates on its limit ideal by Lemma \ref{LemConvAu}.$i$. This limit ideal contains $(P(x),p)$ and its height is at most $h-1 \in \{1,2\}$. Therefore $\lim \aid_n=\aid$ and hence $\lim (G(\lambda_n),S_n)=M(\aid)$.

$(iv)$: Let $\aid=(P(x,y))$ where $P \in \Z\br{x,y}$ is an irreducible polynomial. By $(iii)$, we can assume that $\aid \cdot \Aut(F_2)$ does not contain any polynomial $\Phi_m(x)$, i.e. $P(x,y)$ is not of the form $x^{\alpha}y^{\beta}\Phi_m(x^{k}y^{l})$ with $\alpha,\beta,k,l \in \Z$ and $\gcd(k,l)=1$. It follows from \cite[Section 4, Corollary 1 and Theorem 69]{Sch00} that there exists a sequence $(k_n,l_n)$ of coprime integers such that $|k_n|+|l_n|$ tends to infinity and $P(x^{k_n},y^{l_n})$ has a root $\lambda_n$ which is not a root of unity. Let $S_n$ be a generating pair of $G(\lambda_n)$ such that $\sigma(S_n)=(k_n,l_n)$. Possibly taking a subsequence, we can assume that $(G(\lambda_n),S_n)$ converges. By Lemma \ref{LemConvAu}.$i$ its limit is of the form $M(\aid')$ where $\aid'$ is the limit of the corresponding sequence $(\aid_n)$ of ideals. As $P \in \aid'$ by construction, it suffices to show that $\aid'$ has height $1$. It holds if $(\aid_n)$ has a constant subsequence by Theorem \ref{ThConv}.$ii$. It also holds if $(\aid_n)$ accumulates on $\aid'$ by \cite[Corollary 3.17]{CN03}. Therefore $\aid'=\aid$ and hence $M(\aid)$ is a an accumulation point of groups in $\mathcal{C}_0$. As the remaining cases were addressed in $(iii)$, the proof of the first claim is complete.

Let us prove $(1)$. By Lemma \ref{LemConvAu}.$iii$ and Proposition \ref{PropRoot}.$ii$, the marked groups $M(y-1)=G(0)$ and $M(\Phi_m(x),y-1)$ are accumulation points of sequences of groups $G(\lambda_n)$ where each $\lambda_n$ is a root of unity. So is $M(\Phi_m(x))$ by the proof of $(iii)$ and $M$ by Theorem \ref{ThConv}.$iii$.

Let $(G(\lambda_n),S_n)$ be a sequence which accumulates on its limit $(L,S)$ where all the $\lambda_n$ are roots of unity. Let $(k_n,l_n)=\sigma(S_n)$ and let $d_n$ be the degree of $\lambda_n$ over $\Q$. By Lemma \ref{LemConvAu}.$iii$ and Proposition \ref{PropRoot}.$i$, we can suppose that $d_n$ and $|k_n|+|l_n|$ tends to infinity. By Lemma \ref{LemConvAu}.$i$, the limit $G$ is of the form $M(\aid)$ where $\aid$ is the limit of the corresponding sequence $(\aid_n)$ of ideals. As $d_n$ tends to infinity, $(\aid_n)$ accumulates on $\aid$ whose height is consequently at most $1$. We can suppose that this height $1$, i.e. $\aid=(P(x,y))$ with $P$ irreducible. Reasonning by contradiction, we assume that $\aid \cdot \Aut(F_2)$ does not contain any polynomial of the form $\Phi_m(x)$. Since $P(\lambda_n^{k_n},\lambda_n^{l_n})=0$ for all $n$, it follows from \cite[Section 5, Lemma 6]{Sch00} that there exists two linearly independant vectors $v_1=(\alpha_1,\beta_1),\,v_2=(\alpha_2,\beta_2)$ in $\Z^2$ and a constant $c \in \N \setminus \{0\}$ such that $d_n$ divides $c \cdot \gcd(\alpha_1 k_n+\beta_1 l_n,\alpha_2 k_n+\beta_2 l_n)$ for all $n$. Replacing $(G(\lambda_n),S_n)$ by $(G(\lambda_n),S_n) \cdot \psi$ for some Nielsen transformation $\psi$, we can assume that $\alpha_1=0$. As $\gcd(k_n,l_n)=1$ and $d_n$ tends to infinity, we deduce that $\beta_1=0$, a contradiction. Therefore, we can find $\psi$ and $m \ge 1$ such that $\aid \cdot \psi=(\Phi_m(x))$.

Let us prove $(2)$. We deduce from the proof of $(iii)$ that any marked group $M(P(x,y))$, with $P$ reciprocal and irreducible, is an accumulation point of a sequence of groups $G(\lambda_n)$ where each $\lambda_n$ is a reciprocal algebraic number. Clearly, so are the $\Phi$-limits.

Consider now a sequence $(G(\lambda_n),S_n)$ which accumulates on its limit $(L,S)$ and where all the $\lambda_n$ are reciprocal algebraic numbers. Let $(k_n,l_n)=\sigma(S_n)$. By Lemma \ref{LemConvAu}.$iii$ and Theorem \ref{ThConv}.$ii$, we can suppose that $|k_n|+|l_n|$ tends to infinity and that the corresponding sequence $(\aid_n)$ of ideals accumulates on its limit $\aid$. Consequently the height of $\aid$ is at most $1$ and the limit $G$ is of the form $M(\aid)$ by Lemma \ref{LemConvAu}.$i$. We can suppose that the latter height is $1$, i.e. $\aid=(P(x,y))$ with $P$ irreducible. Since $P(\lambda_n^{k_n},\lambda_n^{l_n})=0$ for all $n$ and $|k_n|+|l_n|$ tends to infinity while $\gcd(k_n,l_n)$=1, we deduce from \cite[Section 6, Lemma 11]{Sch00} that $P$ is reciprocal.

\end{sproof}

\subsection{Tame subsets} \label{SubSecTame}

Let $X \subset \C \setminus \{0\}$. We set \begin{equation} \label{EqDefL} \mathcal{L}(X)=\bigcup_{\lambda \in X} \overline{\br{\G}}.\end{equation}

It follows from Lemma \ref{LemConvAu}.$iii$ that $\Z \wr \Z$ is a limit of some $G(\lambda)$'s with $\lambda \in X$ if and only if $X$ is infinite or contains a transcendant number over the rationals.
Hence the set $\mathcal{L}(X) \cup \br{\Z \wr \Z}$ is the set of 
``obvious limits'' associated to an infinite set $X$. The subset $X$ is said to be \emph{tame} if either $X$ is finite or 
$$\overline{\br{\G:\lambda \in X}} = \mathcal{L}(X) \cup \br{\Z \wr \Z}.$$
Note that tameness is preserved under finite union and that limits associated to a tame set are described by Proposition \ref{PropRoot}.$ii$ and Theorem \ref{ThConv}.$ii$. Note also that Theorem \ref{ThLimMa}.$iv$ provides examples of non-tame sets.

Let $\alpha$ be an algebraic integer over the rationals with conjugates $\lambda_1=\lambda,\lambda_2,\dots,\lambda_n$. Define
$\begin{array}{|c|} \hline \lambda \end{array} =\max_{1 \le i \le n}|\lambda_i|$ where $|\cdot|$ is the complex modulus. 
\begin{theorem} \label{ThTame}
Let $X \subset \C \setminus \{0\}$. If there is $\epsilon>0$ such that $\begin{array}{|c|} \hline \lambda \end{array}>1+\epsilon$ for every algebraic unit $\lambda \in X$, then $X$ is tame.
\end{theorem}

In particular, the set of Pisot numbers \cite{Sal63} and the set of totally real algebraic integers \cite[Section 6, Corollary 7]{Sch00} are tame sets. In addition, we deduce from \cite[Th.1]{SZ65} the following:

\begin{corollary} \label{CorBounded}
Let $X \subset \C \setminus \{0\}$. If there is $d \ge 0$ such that every algebraic unit 
in $X$ has degree at most $d$ over $\Q$, then $X$ is tame.
\end{corollary}

\begin{sproof}{Proof of Theorem \ref{ThTame}}
Let $(G,S)=\lim (G(\lambda_n),S_n)$ be a limit point of $\br{\G : \lambda \in X}$. We set $(k_n,l_n)=\sigma(S_n)$.
If $(\lambda_n)$ has a subsequence of transcendant numbers or a subsequence of pairwise conjugate algebraic numbers, then $(G,S) \in \mathcal{L}(X)$ by Theorem \ref{ThConv}. It also holds if $(k_n,l_n)$ has a constant subsequence by Lemma \ref{LemConvAu}.$iii$. Thus we can assume that the $\lambda_n$ are pairwise non-conjugate algebraic numbers and that $|k_n|+|l_n|$ tends to infinity. Now, suppose that there are infinitely many $\lambda_n$ which are not algebraic units. For every such $\lambda_n$, there is a normalized non-archimedean absolute value $|\cdot|_n$ on $\Q(\lambda_n)$ such that either $|\lambda_n|_n<1/2$ or $|\lambda_n|_n>3/2$ holds. We deduce from Proposition \ref{PropCrit} and Lemma \ref{LemConvAu} that $(G(\lambda_n),S_n)$ has a convergent subsequence with limit $M \in \mathcal{L}(X)$. Hence we can assume that $\lambda_n$ is an algebraic unit for every $n$. By hypothesis, we can find an archimedean absolute value $|\cdot|_n$ on $\Q(\lambda_n)$ such that $|\lambda_n|_n>1+\epsilon$ for all $n$. Again, it follows from Proposition \ref{PropCrit} and Lemma \ref{LemConvAu} that $(G(\lambda_n),S_n)$ has a convergent subsequence with limit $M \in \mathcal{L}(X)$. The proof is then complete.
\end{sproof}

\begin{remark}
It readily follows from \cite[Lemma 3.13]{CN03} that groups in $\mathcal{C}_0$ are limits of groups of the form $\Z_p \rtimes \Z$. 
In particular, $\bigcup_{p>0} \Cp$ is dense in $\Cu$.
Note that Theorems \ref{ThA},\ref{ThB} and \ref{ThC} could be rephrased in a purely ring theoretic setting as they actually
describe special subsets of the closure of $$\Ru\br{y^{\pm 1}} \cdot \Aut(F_2)$$ in $\Rd$, $\Ru\br{y^{\pm 1}}$ being the set of ideals of the form $\bid \br{y^{\pm 1}}$ where $\bid$ is any ideal of $\Zx$.

\end{remark}

\section{Universal theory} \label{SecUniv}

In this section, we prove Theorem \ref{ThD}. 
Let $\lambda \in \C \setminus \{0\}$.
Recall that
$$\Gl=\Z\br{\lambda^{\pm 1}} \rtimes \Z$$ where $\Z$ acts by the multiplication of $\lambda$ on $\Z\br{\lambda^{\pm 1}}$.
We denote by $P_{\lambda}(x) \in \Z\br{x}$, the minimal polynomial of $\lambda$ over $\Z$, agreeing that $P_{\lambda}=0$ if $\lambda$ is transcendant over $\Z$. With the notation of Section \ref{SecLimits}, we have $G(\lambda)=G(P_{\lambda}(x))$.
We set $M(\lambda)=M(P_{\lambda}(x))$ if $\lambda$ is not a root of unity and $M(\lambda)=M(P_{\lambda}(x), y-1)$ else.

\begin{lemma} \label{LemThGl}
Let $\lambda,\mu \in \C \setminus \{0\}$. The equality $\Thu(\Gl)=\Thu(G(\mu))$ holds of and only if one of the following holds:
\begin{itemize}
\item $\lambda$ and $\mu$ are both transcendant over $\Z$,
\item $\lambda$ and $\mu$ are both algebraic over $\Z$ and $\lambda$ is conjugated to $\mu$ or $\mu^{-1}$.
\end{itemize}
\end{lemma}

\begin{proof}
If $\lambda$ and $\mu$ are both transcendant over $\Z$, then the map $a \mapsto a,\, b \mapsto b$ induces a group isomorphism from $\Gl$ onto $\Gm$.
If $\lambda$ and $\mu$ are both algebraic over $\Z$ and $\lambda$ is conjugated to $\mu^{\epsilon}$ with $\epsilon \in \{\pm 1\}$, then the map $a \mapsto a,\, b \mapsto b^{\epsilon}$ induces a group isomorphism from $\Gl$ onto $\Gm$. Thus the universal theories of $\Gl$ and $\Gm$ coincide in each case.

Assume now that $\Thu(\Gl)=\Thu(G(\mu))$. If $\lambda=1$, i.e. $\Gl$ is abelian, then clearly so is $\Gm$, which forces $\mu=1$. Thus we can assume that neither $\lambda$ nor $\mu$ is $1$. We can also assume that one of $\lambda$ and $\mu$ is algebraic, say $\lambda$. Let $\phi_{\lambda}$ be the existential sentence
$$\exists x,\exists y,\exists z \,(\br{\br{x,y},z} \neq 1 \wedge \br{x,y}^{P_{\lambda}(z)}=1)$$ and define $\phi_{\mu}$ analogously. Choosing $x=b,y=z=a$, we easily check that $\phi_{\lambda} \models \Gl$. As $\The(\Gl)=\The(\Gm)$, we have $\phi_{\lambda} \models G(\mu)$. Since $\br{x,y}$ defines a non-zero element of $\Z \br{\mu^{\pm 1}}$ for every $x,y \in \Gm$, we deduce from $\phi_{\lambda}$ that $P_{\lambda}(\mu^k)=0$ in $\C$ with $k=\sigma(z) \neq 0$. Therefore $\mu$ is algebraic over $\Z$ and $\lambda$ is conjugated to $\mu^k$. Reasonning similarly with $\phi_{\mu}$ we obtain that $\mu$ is a conjugate of $\lambda^l$ for some $l \in \Z$.
Consequently $\lambda$ is a conjugate of $\lambda^{kl}$. The result follows from Lemma \ref{LemConj}.
\end{proof}

\begin{lemma} \label{LemGlandMl}
Let $\lambda \in \C \setminus \{0,1\}$.
We have $\Thu(\Gl)=\Thu(\Ml)$.
\end{lemma}
\begin{proof}
We easily check from the definitions that the map $a \mapsto a,\, b \mapsto c$ induces an injective group homomorphism from $\Gl$ into $\Ml$. As a result, $\Thu(\Ml) \subset \Thu(\Gl)$. By Theorem \ref{ThConv}.$ii$, the marked group $\Ml$ is a limit of marked groups isomorphic to $\Gl$. Thus $\Thu(\Gl) \subset \Thu(\Ml)$ by \cite[Proposition 5.2]{ChGu05}, which completes the proof.
\end{proof}

\begin{theorem} \label{ThUniv}
Let $\lambda \in \C \setminus \{0,1\}$ and let $G$ be a non-abelian $2$-generated group. Then the equality $\Thu(G)=\Thu(\Gl)$ holds if and only if $G$ is isomorphic to $\Gl$ or $\Ml$.
\end{theorem}

\begin{proof}
If $G$ is isomorphic to $\Gl$ or $\Ml$, then $\Thu(G)=\Thu(\Gl)$ by Lemma \ref{LemGlandMl}.

Assume now that $\Thu(G)=\Thu(\Gl)$ holds and let $S$ be a generating pair of $G$.
By \cite[Proposition 5.3]{ChGu05}, there is a sequence $(G_n,S_n)$ of $2$-generated marked subgroups of $\Gl$ that converges to $(G,S)$ in $\M$.
As $G$ is not abelian, $G_n$ is a non-abelian $2$-generated subgroup of $\Gl$ for all $n$ large enough. Thus $G_n=G(\lambda^{k_n})$ with $k_n \in  \Z$ for all $n$ large enough.  As $\{\lambda^{k_n}: n \in \N\}$ is tame by Corollary \ref{CorBounded}, $G$ is either isomorphic to a group of the form $G(\lambda^k),M(\lambda^k),M$ ($k \in \Z \setminus \{0\}$) by Proposition \ref{PropRoot}.$ii$ and Theorem \ref{ThConv}.$ii$ or it is isomorphic to $\Z \wr \Z$. It follows from Lemma \ref{LemThGl} and Lemma \ref{LemGlandMl} that $G$ is either isomorphic to $\Gl$ or $\Ml$.

\end{proof}

All marked groups considered in this article are finitely generated metabelian groups, and hence equationally Noetherian groups \cite[Theorem B1]{BMR99}.
The Krull dimension of a finitely generated equationally Noetherian group $G$ is the largest integer $k \ge 0$ such that there exists a chain of surjective homomorphisms $$G=Q_0 \longrightarrow Q_1 \longrightarrow \dots \longrightarrow Q_k$$ where $Q_{i+1}$ is a proper quotient of $Q_i$ with the same universal theory as 
$G$ for every $i$. 
Remeslennikov and Timoshenko calculated the Krull dimensions of finitely generated free metabelian groups and wreath products of free abelian groups of finite rank \cite{RT06}. Myasnikov and Romanovskiy defined a larger class of groups, named \emph{rigid groups}, and showed that the Krull dimension of a rigid group is finite \cite{MR08}. 
The next corollary provide examples of non-rigid groups of Krull dimension one.

\begin{corollary}\label{CorKrull}
Let $\lambda \in \C \setminus \{0,1\}$. Then the Krull dimension of $\Ml$ is $1$. 
\end{corollary}

\begin{proof}
It follows from the definitions and Lemma \ref{LemDisc}.$iii$ that the map $a \mapsto a,b \mapsto b$ induces a surjective morphism $\pi$ from $M(\lambda)$ onto $\Gl$ with non-trivial kernel. 
As $M(\lambda)$ and $\Gl$ are Hopfian, we deduce from Theorem \ref{ThUniv} that $M(\lambda) \stackrel{\pi}{\longrightarrow} \Gl$ is the longuest chain of proper epimorphisms starting with $M(\lambda)$ such that every group in the chain has the same universal theory.
\end{proof}

\section{Acknowledgments}
We are grateful to Tanguy Rivoal and Jean-Louis Verger-Gaugry for useful discussions and references on algebraic integers. We would like to thank Pierre de la Harpe for his comments on an earlier version.

\newcommand{\etalchar}[1]{$^{#1}$}
\def\cprime{$'$}

Author address:

\medskip

L. G. Mathematisches Institut, Georg-August Universit\"at,
Bunsenstrasse 3-5, G\"ottingen 37073, Germany,
guyot@uni-math.gwdg.de
\medskip

\end{document}